\documentclass[11pt,hyp,]{nyjm}
\usepackage[utf8]{inputenc}

\title[Divisorial and geometric gonality]{Divisorial and geometric gonality of higher-rank tropical curves}
\author[J. van Dobben de Bruyn, D. Holmes, D. van der Vorm]{Josse van Dobben de Bruyn, David Holmes, and David van der Vorm}

\email{holmesdst@math.leidenuniv.nl}
\email{w.d.van.der.vorm@umail.leidenuniv.nl}
\email{j.vandobbendebruyn@tudelft.nl}

\usepackage{mathtools}
\usepackage{graphicx}
\usepackage{amsmath}
\usepackage{amsthm}
\usepackage{amssymb}
\usepackage{pstricks}
\usepackage{tikz}   
\usepackage{tikz-cd}
\usetikzlibrary {positioning}
\usepackage{mathrsfs}
\usepackage[hidelinks]{hyperref}
\hypersetup{nesting=true,debug=true,naturalnames=true}
\usepackage{upref}

\usetikzlibrary{decorations.pathreplacing}

\newcommand{\comment}[1]{...}

\newcommand{\Z}{\mathbb{Z}}
\newcommand{\N}{\mathbb{N}}
\newcommand{\R}{\mathbb{R}}
\renewcommand{\phi}{\varphi}

\newcommand{\Div}{\operatorname{Div}}
\newcommand{\Prin}{\operatorname{Prin}}
\newcommand{\rnk}{r}
\newcommand{\G}{\Gamma}

\newcommand{\emp}[1]{\textbf{\emph{#1}}}
\renewcommand{\:}{\colon}

\renewcommand{\>}{\rangle}

\newcommand{\gon}{\operatorname{dgon}}
\newcommand{\ggon}{\operatorname{ggon}}
\newcommand{\PL}{\operatorname{PL}}
\newcommand{\tw}{\operatorname{tw}}
\newcommand{\sn}{\operatorname{sn}}

\newtheorem{theorem}{Theorem}[section]       
\newtheorem{corollary}  [theorem]{Corollary}
\newtheorem{lemma}      [theorem]{Lemma}     
\newtheorem{proposition}[theorem]{Proposition}

\newtheorem*{theorem*}{Theorem}
\newtheorem*{corollary*}{Corollary}
\theoremstyle{definition}
\newtheorem{definition} [theorem]{Definition}
\newtheorem{remark}     [theorem]{Remark}
\newtheorem{example}    [theorem]{Example} 

\begin{document}

\begin{abstract}
    We consider a variant of metrised graphs where the edge lengths take values in a commutative monoid, as a higher-rank generalisation of the notion of a tropical curve. Divisorial gonality, which Baker and Norine defined on combinatorial graphs in terms of a chip firing game, is extended to these monoid-metrised graphs. We define geometric gonality of a monoid-metrised graph as the minimal degree of a horizontally conformal, non-degenerate morphism onto an monoid-metrised tree, and prove that geometric gonality is an upper bound for divisorial gonality in the monoid-metrised case. We also show the existence of a subdivision of the underlying graph whose gonality is no larger than the monoid-metrised gonality. We relate this to the minimal degree of a map between logarithmic curves. 
\end{abstract}

\maketitle

\tableofcontents

\section{Introduction}

Tropical geometry allows us to exchange information between geometric and combinatorial worlds. 
A fundamental example is that of \emph{gonality}. The older notion is the gonality of a smooth curve, defined as the minimal degree of a map to the projective line. More recently, Baker and Norine defined the \emph{divisorial gonality} of a graph in terms of a certain chip-firing game (see \cite{BakerNorine2007RR}, or Definition \ref{d_divgon}). The connection between these runs via degenerations; a family of smooth curves may degenerate to a stable curve, allowed mild (nodal) singularities. To a stable curve we can associate a graph (with a vertex for each irreducible component and an edge for each node), and the gonality of the curve is then closely related to the gonality of the associated graph. This relation has been exploited by \cite{BakerNorine2007RR} and \cite{CoolsDraismaPayneRobeva2012} to prove a combinatorial Riemann-Roch Theorem, and to give a tropical proof of the Brill-Noether Theorem. 

It is clear that replacing a stable curve by its graph forgets a lot of information. One can hope by decorating the graph with some of that information to obtain a tighter relationship between the notions of gonality coming from algebraic geometry and those coming from combinatorics; for example, one could remember the genera of the vertices. In this brief note we explore the consequences of carrying across another kind of geometric data, namely assigning lengths to the edges of the graph. For 1-parameter degenerating families this is already well-explored; given a singularity with local equation $xy - t^n$ we assign the corresponding edge the length $n$. This leads to the classical theory of tropical curves as graphs with edge lengths in $\R_{>0}$, as treated for example in \cite{AminiCaporaso2011, Caporaso2014, CoolsDraisma2018}.

However, for many applications (in particular to moduli spaces of curves) it is natural to consider families of curves over higher-dimensional bases, for which there is no natural integer-valued edge length. Following \cite{Molcho2018} we instead equip our graphs with a length function taking values in a \emph{monoid}, which is the natural setting for a metric from the perspective of logarithmic geometry. Taking the monoid to be $\N$ recovers the classical integer-valued notion of edge-length, or we can consider the monoid to be $\R_{\ge 0}$ to recover (a finite model of) an abstract tropical curve%
\footnote{We warn the reader that the edges of our `monoid-metrised graphs' are discrete objects equipped with edge lengths, so there are no points on the interior of any edge. This differs from the usual notion of a `metric graph', where every edge is an interval containing uncountably many points.}.
However, for us the most interesting cases are where the monoid has higher rank (such as $\N^2$).

We define divisorial gonality on these `monoid-metrised' graphs in the natural way as the minimal degree of a divisor of positive rank. Then we extend the notion of morphisms to monoid-metrised graphs, and define harmonicity and nondegeneracy for these morphisms. We define geometric gonality of a monoid-metrised graph to be the minimal degree of a harmonic, non-degenerate morphism onto a tree.
Our first main result is the following.

\begin{theorem*} [Theorem \ref{t_gonggon}]
    Let $M$ be a sharp, integral monoid and $\G = (X,r,i,l)$ be an $M$-metrised graph. Then $\gon(\G) \leq \ggon(\G)$. 
\end{theorem*}

We will see that this result is an improvement over the classical inequality between geometric and divisorial gonality, at least when one is interested in using these results to bound from below the gonality of an algebraic curve. For such applications one is generally interested in the minimum of gonalities of subdivisions of the underlying combinatorial graph, and our second main result shows that this minimum is no greater than our `monoid-metrised gonality'. 

\begin{theorem*}[Theorem \ref{thm:combinatorial-dgon}]
	Let $\Gamma$ be a monoid-metrised graph, and let $G$ be the underlying graph.
	Then there is a subdivision $H$ of $G$ such that $\gon(\Gamma) \geq \gon(H)$.
\end{theorem*}

Example \ref{rem:infinite_gonality} shows that this inequality is in general strict, as $\gon(\Gamma)$ can be infinite. 

An immediate consequence is that two combinatorial lower bounds on the gonality (namely, \emph{treewidth} \cite{VanDobbenDeBruynGijswijt2020} and \emph{scramble number} \cite{HarpJacksonJensenSpeeter2022}) carry over to monoid-metrised graphs; see Corollary \ref{cor:tw-sn}.


Finally, in Appendix \ref{sec:log_appendix} we give a precise connection between the combinatorial and algebro-geometric gonalities via log geometry. We show that the degree of a flat morphism of log curves is equal to the degree of the associated harmonic morphism of monoid-metrised graphs, from which we deduce that the algebro-geometric gonality is greater than or equal to the geometric gonality of the associated monoid-metrised graph; see Theorem \ref{log_theorem} for a precise statement.
Using this, we show that monoid-metrised gonality gives a sharper bound on geometric gonality than was possible with previous approaches; see Remark \ref{log_eg}.

%

\subsection{Attribution}
Theorem \ref{t_gonggon} was proven by DvdV as part of his bachelor's thesis under the supervision of DH.
After this paper was posted to the arXiv, JvDdB explained to us how to prove Theorem \ref{thm:combinatorial-dgon}, and joined the paper as a co-author. 
The initial version of the text is due to DvdV, with edits by all of us.
Section \ref{sec:comb_appendix} is written by JvDdB, and the appendix is written by DH.

\section{Preliminaries} \label{c_prelim}



\subsection{Monoids}

This subsection is based on Chapter I.1 of \cite{Ogus}.

A \emp{monoid} is a set $M$ with a commutative associative binary operation and an identity element $0\in M$. We say $M$ is \emp{sharp} if the only invertible element is $0$. The \emp{groupification} $M \to M^{gp}$ is the initial object in the category of maps from $M$ to groups; we say $M$ is \emp{integral} if the groupificiation map is injective; equivalently if for every $a,b,c \in M$ we have $a+c = b+c \implies a = b$. 

From now on, by \emp{monoid} we always mean a sharp, integral monoid; key examples to keep in mind are $\N$, $\R_{\ge 0}$, $\N^2$. 

\begin{lemma} \label{l_divisibility}
Let $M$ be a sharp, integral monoid and $m, n \in M$. Suppose there exist positive integers $a$ and $b$ with $an = m$ and $bn = m$. Then $a=b$ or $m=n=0$.
\end{lemma}

In other words, if $an=m$, then $\frac{m}{n} = a$ is well-defined.


\subsection{Monoid-metrised graphs} \label{s_metricgraphs}
Here we present a purely combinatorial version of the tropical curves of \cite{Molcho2018}. Our graphs are connected, undirected, and are allowed self-loops and multiple edges. More formally we define

\begin{definition} \label{d_graph}
    A \emp{graph} is a tuple $G = (X,r,i)$ where
    \begin{itemize}
        \item $X$ is a finite set
        \item $r\: X \to X$ is a idempotent function
        \item $i\: X \to X$ is an involution
        \item $i(x) = x$ if and only if $r(x) = x$
    \end{itemize}
\end{definition}

The subset of $X$ of elements with $i(x) = x$ is denoted $V(G)$ or $V$ (\emp{vertices}). The other elements of $X$ are \emp{half-edges}, which form the set $H = X\setminus V$. We also write $E = \{\{e,i(e)\}: e\in H\}$ for the set of \emp{edges}. 
For $v\in V$, let $H_v = \{e\in H: r(e)=v\}$ be the set of half-edges \emp{adjacent} to $v$. For vertices $u,v\in V$, the set of edges between them is $E(u,v) = \big\{\{e,f\} \in E: \{r(e),r(f)\} = \{u,v\}\big\}$. 
We only consider connected graphs from now on.

\begin{definition} \label{d_length}\label{d_metricgraph}
    Let $M$ be a monoid and $G = (X,r,i)$ a graph. A \emp{metric on $\G$ with values in $M$} is a function $l\: X \to M$ such that
    \begin{itemize}
        \item $l(i(x)) = l(x)$ for all $x\in X$
        \item $l(x) = 0$ if and only if $x \in V(G)$
    \end{itemize}
\end{definition}
We say $\G$ is an \emp{$M$-metrised graph}, or that \emp{$\G$ is metrised by $M$}. 

\begin{remark}
We can simply view a metric over $M$ as a function assigning to every edge a non-zero element of $M$. In particular, if $M = \mathbb R_{\ge 0}$, then this is exactly (a finite model of) an abstract tropical curve. 
The more intricate formulation above via the set $X$ will be convenient when we begin to study morphisms of graphs. 
\end{remark}




\begin{definition}
    Let $\G$ be an ($M$-metrised) graph. A \emp{divisor} $D$ on $\G$ is an element $\sum_{v \in V(\G)} D(v) [v]$ of the free abelian group $\Div(\G)$ on $V(\G)$.
\end{definition}
We often think of $D$ as a function $V \to \mathbb Z$. 
The \emp{degree} of a divisor $D$ is $\deg(D) = \sum_{v\in V} D(v)$. The set of divisors on $\G$ with degree $k$ is denoted by $\Div^k(\G)$. 

We define a partial order $\geq$ on $\Div(\G)$, such that $D \geq D'$ if and only if for all $v\in V$ we have $D(v) \geq D'(v)$. A divisor $D$ is called \emp{effective} if $D\geq 0$. The set of effective divisors on $\G$ is denoted by $\Div_+(\G)$. Similarly, we define $\Div_+^k(\G) = \Div_+(\G) \cap \Div^k(\G)$, the set of effective divisors of degree $k$.

\begin{definition} \label{d_PL}
    Let $\G$ be an $M$-metrised graph. The set of \emp{piecewise linear functions} on $\G$ is
    $$ \PL(\G) = \{ g\: V(\G) \to M^{gp}: \forall e\in H: g(r(e))-g(r(i(e))) \in \langle l(e) \rangle \}; $$
\end{definition}
\noindent in other words, we require that the difference between the values of $g$ at the ends of an edge is an integer multiple of the length of that edge. 
Piecewise linear functions on $M$-metrised graphs can be seen as a discrete equivalent of \emph{rational} functions on algebraic varieties. 

Let $e\in H$, $u=r(e), v=r(i(e))$. If $g(u)-g(v) \in \langle l(e) \rangle$, then $g(u)-g(v)$ is an integer multiple of $l(e)$, and Lemma \ref{l_divisibility} ensures that $\frac {g(u) - g(v)} {l(e)}$ is well defined, and it is an integer. 

\begin{definition} \label{d_deltam}
    Let $\G$ be an $M$-metrised graph. The \emp{Laplacian} $\Delta$ is the group homomorphism
    $$ \Delta \: \PL(\G) \to \Div(\G) \: $$
    $$ g \to \sum_{v\in V}\left( \sum_{e\in H_v} \frac {g(v) - g(r(i(e)))} {l(e)} [v]\right) 
    $$
\end{definition}

\begin{lemma} \label{l_deltamker}
    The kernel of $\Delta$ is $\ker(\Delta) = \{g\in \PL(\G) : g \text{ is constant}\}$.
\end{lemma}

\begin{example} \label{e_PLconstant}
    \noindent Consider the following graph $\G$ over $\N^2$: 
    
    
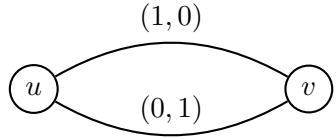
\begin{center}
        \begin {tikzpicture}[auto, node distance=3cm, every loop/.style={},
                        thick,main node/.style={circle,draw,font=\sffamily\bfseries}]
        \node[main node] (1) [] {$u$};
        \node[main node] (2) [right=of 1] {$v$};
        
        \path (1) edge [out=30, in=150] node {$(1,0)$} (2);
        \path (1) edge [out=330, in=210] node {$(0,1)$} (2);
        \end{tikzpicture} 
    \end{center}
    
    \noindent Then $g\in \PL(\G)$ would have to satisfy $g(v)-g(u) = r\cdot (1,0) = s\cdot (0,1)$ for some $r,s\in \mathbb Z$, and thus $r=s=0$ and $g(v) = g(u)$. Hence $\PL(\G)$ consists of all constant functions $V\to \mathbb N^2$, and $\Delta$ is the zero map here.
\end{example}

The divisors in the image of $\Delta$ are called \emp{principal divisors}. The corresponding subgroup is denoted by $\Prin(\G) \coloneqq \Delta[\PL(\G)] \subset \Div^0(\G)$.

The equivalence relation $\sim_\Delta$ on $\Div(\G)$ with $D\sim_\Delta D' \iff D-D'\in \Prin(\G)$ is called \emp{linear equivalence}.
The equivalence class of $D \in \Div(\G)$ under $\sim_\Delta$ is denoted with $[D]$. 
The \emp{linear system associated to $D$} is $\{E \in [D] : E \geq 0\}$.
The \emp{rank} or \emp{dimension} of $D$ is 
$$ \rnk(D) = \max \{ k \in \Z: \text{for all } F \in \Div_+^k(\G), |D-F| \neq \emptyset \} $$

\begin{definition} \label{d_divgon}
    Let $\G$ be an $M$-metrised graph. The \emp{divisorial gonality} of $\G$ is 
    $$ \gon(\G) = \min\{ \deg(D): D \in \Div(\G), r(D) \geq 1 \} $$
\end{definition}

\begin{remark}
    In the language of the Baker-Norine chip firing game, the divisorial gonality of $\G$ is the smallest degree of a \emph{winning} divisor.
\end{remark}
\vspace{11pt}


\subsection{Harmonic morphisms} \label{ss_harmmor} 

A morphism of $M$-metrised graphs should take vertices to vertices, and edges to edges or vertices, satisfying the obvious compatibilities. If an edge $e$ sent to an edge $e'$ then we require that the length of $e'$ is an integer multiple of the length of $e$. This is formalised in the following definition.

 
\begin{definition} \label{d_morpismm}
    Let $\G = (X,r,i,l)$ and $\G' = (X',r',i',l')$ be $M$-metrised graphs. A \emp{morphism} from $\G$ to $\G'$ is a map $\phi\: X \to X'$ such that 
    \begin{itemize}
        \item for all $v\in V(\G)$, we have $\phi(v) \in V(\G')$
        \item for all $e\in H(\G)$, if $\phi(e) = e'\in H(\G')$ then $\phi(r(e)) = r'(e'), \phi(r(i(e))) = r'(i'(e'))$ and $l'(e') \in \langle l(e) \rangle $
        \item for all $e\in H(\G)$, if $\phi(e) = v' \in V(\G')$ then $\phi(r(e)) = \phi(r(i(e))) = v'$
    \end{itemize}
\end{definition}

We denote a morphism with $\phi\: \G \to \G'$ rather than $\phi\: X(\G) \to X(\G')$.

For any half-edge $e\in H(\G)$ and $e' = \phi(e)$, let $\mu_\phi(e) = l'(e')/l(e)$ denote the \emp{slope} of $\phi$ on $e$. This is a positive integer by Lemma \ref{l_divisibility}. The slope of a half-edge $e$ is zero if and only if $\phi(e) \in V'$. 
For a vertex $v \in V(\G)$ and half-edge $e' \in H_{\phi(v)} \subset H(\G')$, the \emp{multiplicity} of $e' \in H(\G')$ at $v \in V(\G)$ is
$$ m_{\phi,v}(e') = \sum_{e\in H_v: \phi(e) = e'} \mu_\phi(e) $$

Suppose that for a vertex $v \in V$ any two half-edges $e',f'\in H_{\phi(v)}$ adjacent to $\phi(v)$ have the same multiplicity at $v$. We define the \emp{horizontal multiplicity} of $v$ under $\phi$ to be $m_\phi(v) = m_{\phi,v}(e')$ for any $e'\in H_{\phi(v)}$.

A morphism $\phi\: \G \to \G'$ is \emp{horizontally conformal} if the horizontal multiplicity is well defined for each vertex $v\in V$. 
A horizontally conformal morphism $\phi\: \G \to \G'$ is \emp{non-degenerate} if $m_\phi(v) > 0$ for all $v\in V$.

The \emp{multplicity} of a half-edge $e' \in H(\G')$ is 
$$ m_\phi(e') = \sum_{v \in V(\G): \phi(v)=r'(e')} m_{\phi,r(e)}(e').  $$

That is, the multiplicity of $e'$ is the sum of the multiplicities of $e'$ at $v$ for the vertices $v$ that are mapped onto the root $r'(e')$ of $e'$. It is clear that $m_\phi(e') = m_\phi(i(e'))$. Also, if $\phi$ is horizontally conformal, then for any $f'\in H_{r'(e')}$ we have
$$ m_\phi(e') = \sum_{v \in V(\G)~:~ \phi(v)=r'(e')} m_{\phi,r(e)}(e')
= \sum_{v \in V(\G)~:~ \phi(v)=r'(f')} m_{\phi,r(f)}(f') = m_\phi(f')
$$
Since $\G$ and $\G'$ are both connected, we have the same equality $m_\phi(e') = m_\phi(f')$ for any two half-edges $e',f'\in H(\G')$. 

\begin{definition}
    Let $\G,\G'$ be $M$-metrised graphs and $\phi\:\G \to \G'$ a horizontally conformal morphism. The \emp{degree} of $\phi$ is $\deg(\phi) = m_\phi(e')$ for any $e'\in H(\G')$.
\end{definition}

Horizontally conformal morphisms are often called `harmonic' (\cite{BakerNorine2007HM}), because they pull back harmonic functions to harmonic functions.
When $\G$ is simple (that is, $|E(u,v)| \leq 1$ for any $u,v \in V(\G)$), the converse also holds: a morphism that pulls back harmonic functions to harmonic functions is horizontally conformal. We follow Baker and Norine (\cite{BakerNorine2007HM}) and call a horizontally conformal morphism \emph{harmonic} from now on.

\subsection{Geometric gonality} \label{ss_ggon}
A (monoid-metrised) graph is called a \emp{tree} if it contains no cycles. 
\begin{definition} \label{d_ggon}
Let $\G$ be an $M$-metrised graph. The \emp{geometric gonality of $\G$} is 
$$ \ggon(\G) = $$
$$ \min \{\deg(\phi): T \text{ a tree, }\phi \: \G \to T \text{ a harmonic, non-degenerate morphism} \} $$ 
\end{definition}

\begin{remark}
    This definition is based on the notion of geometric gonality of combinatorial graphs in \cite{Aidun2018}. Alternatively, geometric gonality may be defined as the minimum degree of an \emph{indexed} harmonic morphism onto a tree (\cite[Def 2.1]{Caporaso2014} and \cite[Def. 1.3.1]{VanderWegen2017}). In that case, divisorial and (indexed) geometric gonality do not bound one another (\cite[Ex. 2.18,2.19]{Caporaso2014} and \cite[Ex. 2.1.3]{VanderWegen2017}). Another version of geometric gonality of $G$ is the minimal degree of a harmonic non-degenerate morphism between a \emph{refinement} of $\G$ and a tree \cite{CoolsDraisma2018}.
\end{remark}

\begin{remark}\label{rem:infinite_gonality}
    Not all $M$-metrised graphs allow a harmonic, non-degenerate morphism onto a tree. If the graph $\G$ of Example \ref{e_PLconstant} is mapped surjectively onto a tree $T$, then $|V(T)| \in \{1,2\}$. If $|V(T)| = 1$, then the morphism is degenerate. If $|V(T)| = 2$, then the two edges with lengths $(0,1)$ and $(1,0)$ have the same image $\{e,f\}$. If the map is a morphism, then $(1,0), (0,1) \in \<l(\{e,f\})\>$, which is not possible. Hence $\ggon(\G) = \infty$.
\end{remark}

Clearly the degree of a harmonic morphism is at least 1, and so is the geometric gonality of an $M$-metrised graph.

\begin{lemma}
    An $M$-metrised graph $\G$ has geometric gonality {\rm 1} if and only if $\G$ is a tree.
\end{lemma}
\vspace{11pt}


\section{Geometric gonality of a monoid-metrised graph}\label{c_gonggon}

The first main theorem of this note is the following inequality between divisorial gonality and geometric gonality. This is a generalisation of the known inequality for combinatorial graphs.

\begin{theorem} \label{t_gonggon}
    Let $M$ be a sharp, integral monoid and $\G = (X,r,i,l)$ be an $M$-metrised graph. Then $\gon(\G) \leq \ggon(\G)$. 
\end{theorem}

The inequality comes down to the following. Given a harmonic, non-degenerate morphism of an $M$-metrised graph $\G$ onto an $M$-metrised tree $T$, we can construct a divisor of positive rank on $\G$. For this, we need to define a pullback map of divisors, a pullback map of piecewise linear functions, and a pushforward map of divisors, corresponding to a morphism $\phi$.

\begin{definition}
    Let $\phi\: \G \to \G'$ be a harmonic morphism of $M$-metrised graphs. Then the \emp{pullback map} of divisors is the group homomorphism 
    $$ \phi^* \: \Div(\G') \to \Div(\G) \: 
    D' \mapsto \sum_{v \in V(\G)}D'(\phi(v)) m_\phi(v)[v]. $$
\end{definition}
It can be verified that $\deg(\phi^*(D')) = \deg(\phi) \cdot \deg(D')$, and that $\phi^*(D')$ is effective if and only if $D'$ is effective.


\begin{lemma}
    Let $\phi\: \G \to \G'$ be a harmonic morphism of $M$-metrised graphs, and let $g$ be a piecewise linear function on $\G'$. Then $g \circ \phi$ is a piecewise linear function on $\G$. 
\end{lemma}
\begin{proof}
    Clearly $g \circ \phi$ is a map from $V$ to $M^{gp}$. 
    Let $g\in \PL(\G')$, and $e\in H$ an half-edge, and $u=r(e)$ and $v=r(i(e))$ the corresponding vertices. Then either $\phi(u)=\phi(v)$ and thus $g(\phi(u)) - g(\phi(v)) = 0 \in \<l(e)\>$, or $\phi(u),\phi(v)$ are connected by an edge $\{e',i'(e')\}$, and we have $g(\phi(u)) - g(\phi(v)) \in \<l'(e')\> \subset \<l(e)\>$. Therefore $g(\phi(r(e))) - g(\phi(r(i(e)))) \in \<l'(e')\> \subset \<l(e)\>$ for any half-edge $e$, and therefore $g\circ \phi \in \PL(\G)$.
\end{proof}

\begin{definition}
    Let $\phi\: \G \to \G'$ be a harmonic, non-degenerate morphism of $M$-metrised graphs. Then the \emp{pullback map} of piecewise linear functions is the group homomorphism 
    $$ \phi^* \: \PL(\G') \to \PL(\G) \: g \mapsto g \circ \phi $$
\end{definition}

Although both pullback maps are denoted with $\phi^*$, it should be clear from the context which one is used. The next lemma shows that pulling back commutes with the Laplacians $\Delta$ and $\Delta'$; the proof is a lengthy verification, and is omitted.

\begin{proposition} \label{p_phi^*}
    Let $\phi\:\G \to \G'$ be a harmonic, non-degenerate morphism of $M$-metrised graphs, and $\Delta, \Delta'$ be the Laplacians corresponding to $\G$ and $\G'$. Then $\phi^* \circ \Delta' = \Delta \circ \phi^*$. 
\end{proposition}

In other words, the following diagram commutes:


\begin{center}
\begin{tikzcd}
\operatorname{PL}(\Gamma') \arrow[dd, "\Delta'"] \arrow[rr, "\varphi^*"] &  & \operatorname{PL}(\Gamma) \arrow[dd, "\Delta"] \\
   &  &     \\
\operatorname{Div}(\Gamma') \arrow[rr, "\varphi^*"] &  & \operatorname{Div}(\Gamma)                    
\end{tikzcd}

\end{center}

Given a harmonic, non-degenerate morphism $\phi\: \G \to \G'$, and a divisor $D'\in \Div(\G')$, the corresponding divisor $\phi^*(D)$ is of most interest. In Lemma \ref{l_phirank}, we prove that $\phi^*$ does not decrease the rank of a divisor, and we use the \emph{pushforward} map for the proof.

\begin{definition}
    Let $\phi\:\G \to \G'$ be a harmonic morphism of $M$-metrised graphs. The \emp{pushforward} map on divisors is the group homomorphism
    $$ \phi_*\: \Div(\G) \to \Div(\G') \: 
    D \mapsto \sum_{v\in V(\G)} D(v) [\phi(v)].  $$
\end{definition}

It is clear that $\deg(\phi_*(D)) = \deg(D)$. We also recall that $\deg(\phi^*(D')) = \deg(\phi) \cdot \deg(D')$ for $D'\in \Div(\G')$. 

\begin{lemma}
    Let $\phi\:\G \to \G'$ be a harmonic, non-degenerate morphism of $M$-metrised graphs, and $F \in \Div_+(\G)$. Then $\phi^*(\phi_*(F)) \geq F$.
\end{lemma}
\begin{proof}
    We have
    $$ \phi^*(\phi_*(F)) = \phi^* \left( \sum_{v\in V(\G)} F(v) [\phi(v)] \right) 
    = \phi^* \left( \sum_{v'\in V(\G')} \left( \sum_{u: \phi(u)=v'} F(u) \right) [v'] \right)$$
    $$ \sum_{v\in V(\G)} \left( \sum_{u: \phi(u)=\phi(v)} F(u) \right) m_\phi(v) [v]
    \geq \sum_{v\in V(\G)} F(v) m_\phi(v) [v] $$
    $$ \geq \sum_{v\in V(\G)} F(v) [v]
    = F $$
    For the first inequality, we used that $F(u) \geq 0$ for all $u \in V(\G)$. For the second inequality, we used that $m_\phi(v) \geq 1$ for all $v \in V(\G)$ as $\phi$ is non-degenerate.
\end{proof}

We can now prove that the pullback of divisors does not decrease the rank of an effective divisor. This means that pulling back an effective divisor of positive rank and degree $d$, gives an effective divisor of positive rank and degree $\deg(\phi) \cdot d$.

\begin{lemma} \label{l_phirank}
    Let $\phi\:\G \to \G'$ be a harmonic, non-degenerate morphism of $M$-metrised graphs. Then for all $D'\in \Div_+(\G')$ we have $\rnk(\phi^*(D')) \geq \rnk(D')$.
\end{lemma}
\begin{proof}
    For any divisor $D \in \Div(\G)$ we have
    $$ r(D) = \max\{ k\in \Z: \forall F\in \Div_+^k(\G): |D-F| \neq \emptyset \} $$
    $$ = \max\{ k\in \Z: \forall F\in \Div_+^k(\G): \exists E \in |D|: E \geq F \} $$
    Let $D'\in \Div_+(\G')$, and $D = \phi^*(D') \in \Div_+(\G)$, and $k = r(D') \geq 0$. 
    We will show that, for any $F'' \in \Div_+^k(\G)$, there is an $E \in |D|$ such that $E \geq F$, which implies $r(D) \geq k = r(D')$. The relations between all divisors involved are summarized in the figure below.

    For any $F'' \in \Div_+^k(\G)$ let $F'=\phi_*(F'') \in \Div_+^k(\G')$. Then there is an $E'\in |D'|$ with $E'\geq F'$ since $r(D') = k$. It is clear that $E = \phi^*(E')$ is effective because $E'$ is effective, and $E'-D' \in \Prin(\G')$ implies $\phi^*(E'-D') = E-D \in \Prin(\G)$ by Proposition \ref{p_phi^*}, so $E \in |D|$. We also have $E'-F' \geq 0$ so $E \geq F = \phi^*(\phi_*(F'')) \geq F''$. 
    
    Hence for all $F''\in \Div_+^{r(D')}(\G)$ there is a $E \in |\phi^*(D')|$ with $E \geq F''$, and thus $r(\phi^*(D')) \geq r(D')$.

    \begin{tikzcd}
    & 0 \arrow[r, white, "{\color{black}\leq}" description]
    & F' \arrow[d, "\phi^*"] \arrow[r, white, "{\color{black}\leq}" description] 
    & E' \arrow[d, "\phi^*"] \arrow[r, white, "{\color{black}\sim}" description] 
    & D' \arrow[d, "\phi^*" ] 
    \\
    0 \arrow[r, white, "{\color{black}\leq}" description] \arrow[ru, leftrightarrow]
    & F'' \arrow[ru, "\phi_*"] \arrow[r, white, "{\color{black}\leq}" description] 
    & F \arrow[r, white, "{\color{black}\leq}" description]            
    & E \arrow[r, white, "{\color{black}\sim}" description]         & D           
    \end{tikzcd}

\end{proof}

Theorem \ref{t_gonggon} follows immediately.

\begin{theorem*} [Theorem \ref{t_gonggon}]
    Let $M$ be a sharp, integral monoid and $\G = (X,r,i,l)$ be an $M$-metrised graph. Then $\gon(\G) \leq \ggon(\G)$. 
\end{theorem*}
\begin{proof}
    If $\ggon(\G) = \infty$, then the inequality holds since $\gon(\G) \leq |V(\G)| < \infty$.
    Otherwise, let $\phi\: \G \to T$ be a harmonic, non-degenerate morphism of $M$-metrised graphs of minimal degree $\deg(\phi) = \ggon(\G)$, where $T$ is a tree. Then for any $v\in V(T)$, the divisor $D'_v = (v)$ has rank $r(D'_v) = 1$. By Lemma \ref{l_phirank}, $D_v = \phi^*(D'_v) \in \Div(\G)$ is an effective divisor of degree $\deg(D) = \deg(\phi) = \ggon(\G)$, and rank $\rnk(D_v) \geq \rnk(D'_v) = 1$. Hence $D_v$ is winning, and $\gon(\G) \leq \deg(D_v) = \ggon(\G)$.
\end{proof}

\let\dgon=\gon

\section{Combinatorial perspective}
\label{sec:comb_appendix}

This section compares the divisorial gonality of a monoid-metrised graph $\Gamma$ to the divisorial gonality of the underlying combinatorial graph $G$.
The main result of this section is that there exists a subdivision $H$ of $G$ such that $\gon(\Gamma) \geq \dgon(H)$.
This confirms that the monoid-metrised gonality is an improvement over the usual notion of divisorial gonality when the goal is to lower bound the gonality of a curve, and shows that two combinatorial lower bounds on the gonality (\emph{treewidth} and \emph{scramble number}) also apply to monoid-metrised graphs.

It is easy to see that Definition \ref{d_graph} is equivalent to the standard notion of multigraphs\footnote{A \emph{multigraph} is a finite graph where loops and parallel edges are allowed.} in combinatorics, and that the divisorial gonality of an $\N$-metrised graph where all edges have length $1$ corresponds to the standard notion of divisorial gonality (as defined in \cite{Baker2008Specialization,VanderWegen2017,VanDobbenDeBruyn2012,VanDobbenDeBruynGijswijt2020}).
Our goal will be to replace the lengths of the edges of a monoid-metrised graph by natural numbers, in such a way that positive rank divisors are preserved.

Assume once again that $M$ and $N$ are sharp, integral monoids.
Following \cite{Molcho2018}, we define an \emph{edge contraction} as follows.

\begin{definition}
	Let $f : M \to N$ be a homomorphism.
	For an $M$-metrised graph $\Gamma$, the \emp{edge contraction} of $\Gamma$ along $f$ is the $N$-metrised graph $\Gamma'$ obtained by applying $f$ to edge lengths of $\Gamma$ and subsequently contracting all edges of length $0$.
	(In other words, $\Gamma'$ is the quotient of $\Gamma$ where we identify $e \sim r(e) \sim i(e) \sim r(i(e))$ whenever $f(l(e)) = 0$. Equipped with the length function $f \circ l$, this becomes an $N$-metrised graph.)
\end{definition}

\begin{definition}
	Let $f : M \to N$ be a homomorphism, and let $\varphi : \Gamma \to \Gamma'$ be the edge contraction of $\Gamma$ along $f$.
	The \emp{pushforward} map on divisors is the group homomorphism
	\[ f_*\: \Div(\Gamma) \to \Div(\Gamma'),\ 
	   D \mapsto \sum_{v \in V(\Gamma)} D(v) [\phi(v)]. \]
\end{definition}

\begin{proposition}
	\label{prop:contraction-PL}
	Let $\Gamma$ be an $M$-metrised graph, let $f : M \to N$ be a homomorphism, let $\varphi : \Gamma \to \Gamma'$ be the edge contraction of $\Gamma$ along $f$, and let $g \in \PL(\Gamma)$.
	Then $f^{gp} \circ g$ is a well-defined, piecewise linear function on $\Gamma'$, and $\Delta(f^{gp} \circ g) = f_*(\Delta(g))$.
\end{proposition}
\begin{proof}
	For every $e \in H(\Gamma)$, there exists some $m \in \Z$ such that $g(r(e)) - g(r(i(e))) = ml(e)$, hence $(f^{gp} \circ g)(r(e)) - (f^{gp} \circ g)(r(i(e))) = mf(l(e))$.
	For edges that are contracted (i.e.{} $f(l(e)) = 0$), this shows that $f^{gp} \circ g$ takes on the same values at both endpoints of that edge, so $f^{gp} \circ g$ is a well-defined function $V(\Gamma') \to N^{gp}$.
	For edges that are not contracted, this shows that the difference between the values of $f^{gp} \circ g$ at the endpoints of that edge is a multiple of the length of the edge, so $f^{gp} \circ g$ is piecewise linear.
	
	It remains to prove that $\Delta(f^{gp} \circ g) = f_*(\Delta(g))$.
	Given a piecewise linear function $g \in \PL(\Gamma)$ and an edge $\{e,i(e)\} \in E(\Gamma)$, we define the \emph{contribution of $\{e,i(e)\}$ to $\Delta(g)$} to be the divisor
	\[ \Delta_{\{e,i(e)\}}(g) = \frac{g(r(e)) - g(r(i(e)))}{l(e)}[r(e)] + \frac{g(r(i(e))) - g(r(e))}{l(e)}[r(i(e))]. \]
	By the above, if $f(l(e)) \neq 0$, then
	\[ \frac{g(r(e)) - g(r(i(e)))}{l(e)} = \frac{(f^{gp} \circ g)(r(e)) - (f^{gp} \circ g)(r(i(e)))}{f(l(e))}, \]
	so for each non-contracted edge $\{e,i(e)\} \in E(\Gamma)$ we have $f_*(\Delta_{\{e,i(e)\}}(g)) = \Delta_{\{\varphi(e),\varphi(i(e))\}}(f^{gp} \circ g)$.
	On the other hand, if $f(l(e)) = 0$, then the endpoints of the edge $\{e,i(e)\}$ are identified with one another, so $f_*(\Delta_{\{e,i(e)\}}(g)) = 0$, because the two non-zero entries of $\Delta_{\{e,i(e)\}}(g)$ are added up and cancel out.
	Therefore we have
	\begin{align*}
		f_*(\Delta(g)) &= f_*\left(\sum_{\{e,i(e)\} \in E(\Gamma)} \, \Delta_{\{e,i(e)\}}(g) \right)\\[.5ex]
		&= \sum_{\{e',i'(e')\} \in E(\Gamma')} \Delta_{\{e',i'(e')\}}(f^{gp} \circ g)\\[1.5ex]
		&= \Delta(f^{gp} \circ g). \qedhere
	\end{align*}
\end{proof}
\begin{lemma}
	\label{lem:contraction-rank}
	Let $\Gamma$ be an $M$-metrised graph, let $f : M \to N$ be a homomorphism, and let $D \in \Div(\Gamma)$.
	Then $r(f_*(D)) \geq r(D)$.
\end{lemma}
\begin{proof}
	Let $\Gamma'$ be the contraction of $\Gamma$ along $f$.
	Write $k = r(D)$, and let $F' \in \Div_+^k(\Gamma')$ be given.
	It is easy to see that $f_*[\Div_+^k(\Gamma)] = \Div_+^k(\Gamma')$, so we may choose some $F \in \Div_+^k(\Gamma)$ such that $f_*(F) = F'$.
	Since $r(D) \geq k$, we have $|D - F| \neq \varnothing$, so we may choose $g \in \PL(\Gamma)$ such that $D - \Delta(g) \geq F$.
	By Proposition \ref{prop:contraction-PL}, we have
	\[ f_*(D) - \Delta(f^{gp} \circ g) = f_*(D - \Delta(g)) \geq f_*(F) = F', \]
	which shows that $|f_*(D) - F'| \neq \varnothing$.
	This holds for all $F' \in \Div_+^k(\Gamma')$, so we have $r(f_*(D)) \geq k = r(D)$.
\end{proof}
\begin{remark}
	\label{rmk:injective-contraction}
	If $f : M \to N$ is injective, then $r(f_*(D)) = r(D)$.
	To see this, note that $\Gamma'$ is now the same monoid-metrised graph, but with weights in a larger monoid $N \supseteq M$.
	It is easy to see that $M^{gp} \subseteq N^{gp}$ and $\Prin(\Gamma) \subseteq \Prin(\Gamma')$.
	Conversely, given $g' \in \PL(\Gamma')$, choose some $v_0 \in V(\Gamma')$, and define $g(v) = g'(v) - g'(v_0)$.
	Then $g$ is piecewise linear and differs from $g'$ by a constant, so $\Delta(g) = \Delta(g')$.
	Moreover, using piecewise linearity of $g'$ and connectedness of $\Gamma$, we see that $g$ takes values in $M^{gp} \subseteq N^{gp}$, so $g \in \PL(\Gamma)$.
	This shows that $\Prin(\Gamma) = \Prin(\Gamma')$, so $r(f_*(D)) = r(D)$.
\end{remark}

The following corollary is immediate from Lemma \ref{lem:contraction-rank}.

\begin{corollary}
	\label{cor:contraction-dgon}
	Let $f : M \to N$ be a homomorphism, and let $\Gamma'$ be the edge contraction of $\Gamma$ along $f$.
	Then $\dgon(\Gamma') \leq \dgon(\Gamma)$.
\end{corollary}

We also need the following lemma, which is a special case of \cite[Ch.~I, Prop.~2.2.1]{Ogus}.

\begin{lemma}
	\label{lem:strictly-positive-functional}
	If $M$ is finitely generated \textup(in addition to being sharp and integral\textup), then there exists a homomorphism $f : M \to \N$ such that $f(m) > 0$ for all $m \neq 0$.
\end{lemma}

We now come to the main result of this section, which compares the divisorial gonality of a monoid-metrised graph with the divisorial gonality of the underlying combinatorial graph.

\begin{definition}
	Let $G$ be a graph.
	\begin{enumerate}
		\renewcommand{\labelenumi}{(\alph{enumi})}
		\item A \emp{subdivision operation} adds an extra vertex on the midpoint of some edge of $G$.\footnote{More formally, in terms of terminology from Definition \ref{d_graph}, we delete an edge $\{e,i(e)\}$ from $G$, and we add a new vertex $w$ and two new edges $\{e',i(e')\},\{e'',i(e'')\}$ to $G$, where $r(e') = r(e)$, $r(i(e')) = r(e'') = w$, and $r(i(e'')) = r(i(e))$.}
		
		\item A \emp{subdivision} of $G$ is a graph $H$ that can be obtained from $G$ after a finite sequence of subdivision operations.
		
		\item A \emp{loop} in $G$ is an edge $\{e,i(e)\}$ such that $r(e) = r(i(e))$.
	\end{enumerate}
\end{definition}

It is easy to see that loops can be freely added and deleted from the graph.
\begin{proposition}
	\label{prop:remove-loops}
	Let $\Gamma$ be an $M$-metrised graph, and let $\Gamma'$ be the $M$-metrised graph obtained from $\Gamma$ by removing all loops.
	Then $\Prin(\Gamma) = \Prin(\Gamma')$, $r_\Gamma(D) = r_{\Gamma'}(D)$ for every $D \in \Div(\Gamma)$, and $\dgon(\Gamma) = \dgon(\Gamma')$.
\end{proposition}

A loopless $\N$-metrised graph can be converted to an unmetrised graph in the following way.
\begin{lemma}
	\label{lem:rank-determining}
	Let $\Gamma$ be an $\N$-metrised graph without loops, and let $G$ be the underlying graph.
	Then $\dgon(\Gamma) \geq \dgon(H)$, where $H$ is the graph obtained from $G$ by subdividing every edge $\{e,i(e)\} \in E(G)$ exactly $l(e) - 1$ times.
\end{lemma}
\begin{proof}
	Let $\iota : V(\Gamma) \hookrightarrow V(H)$ denote the inclusion of vertex sets, and let $\iota_* : \Div(\Gamma) \to \Div(H)$ be the associated pushforward map.
	Let $D \in \Div(\Gamma)$ be a positive rank divisor of minimum degree.
	For every $v \in V(\Gamma)$, choose some $g_v \in \PL(\Gamma)$ such that $D - \Delta(g_v) \geq [v]$.
	Every edge $\{e,i(e)\}$ in $\Gamma$ corresponds to a path of length $l(e)$ in $H$, so $g_v$ can be extended to a piecewise linear function $h_v \in \PL(H)$ by linear interpolation.
	Clearly $\iota_*(\Delta(g_v)) = \Delta(h_v)$, so we have $\iota_*(D) - \Delta(h_v) \geq [v]$ for all $v \in V(\Gamma)$.
	This shows that $\iota_*(D) \in \Div(H)$ reaches all original vertices of $\Gamma$.
	Because $\Gamma$ is loopless, it now follows from \cite[Lem.~2.6]{VanDobbenDeBruynGijswijt2020} that $r(\iota_*(D)) \geq 1$.\footnote{This can also be deduced from Luo's theorem on rank-determining sets \cite[Thm.~1.6]{Luo11}, but this requires more work in going back and forth between $H$ and the associated metric graph (in the combinatorial sense; i.e., the metric space obtained by identifying each edge of $H$ with a unit interval, glued together at their endpoints as prescribed by $H$).}
	Therefore, $\dgon(H) \leq \deg(\iota_*(D)) = \deg(D) = \dgon(\Gamma)$.
\end{proof}

We now have all the ingredients to prove the main theorem of this section.
\begin{theorem}[{Compare \cite[Theorem 5.1]{VanDobbenDeBruynGijswijt2020}}]
	\label{thm:combinatorial-dgon}
	Let $\Gamma$ be an $M$-metrised graph, and let $G$ be the underlying graph.
	Then there is a subdivision $H$ of $G$ such that $\dgon(\Gamma) \geq \dgon(H)$.
\end{theorem}
\begin{proof}
	By Remark \ref{rmk:injective-contraction}, we may assume without loss of generality that $M$ is the monoid generated by the lengths of the edges in $\Gamma$, so that $M$ is finitely generated.
	Furthermore, by Proposition \ref{prop:remove-loops}, we may freely remove loops from $\Gamma$ and put them back after we have found our subdivision $H$.
	So assume without loss of generality that $\Gamma$ is loopless. 
	
	By Lemma \ref{lem:strictly-positive-functional}, we may choose a homomorphism $f : M \to \N$ such that $f(m) \neq 0$ for all $m \neq 0$.
	Let $\Gamma'$ be the contraction of $\Gamma$ along $f$.
	Since $f(l(e)) \neq 0$ for all $e \in E$, no edges are contracted, so $\Gamma'$ is simply the $\N$-metrised graph obtained from $\Gamma$ by relabelling the weights of the edges according to $f$.
	It follows from Corollary \ref{cor:contraction-dgon} that $\dgon(\Gamma) \geq \dgon(\Gamma')$, and from Lemma \ref{lem:rank-determining} that $\dgon(\Gamma') \geq \dgon(H)$ for some subdivision $H$ of $G$.
\end{proof}

An immediate consequence is that the \emph{treewidth} and \emph{scramble number} lower bounds carry over (see \cite{VanDobbenDeBruynGijswijt2020,HarpJacksonJensenSpeeter2022} for definitions).

\begin{corollary}[{Compare \cite[Cor.~5.2]{VanDobbenDeBruynGijswijt2020}, \cite[Thm.~1.1]{HarpJacksonJensenSpeeter2022}}]
	\label{cor:tw-sn}
	Let $\Gamma$ be a monoid-metrised graph, and let $G$ be the underlying graph.
	Then $\dgon(\Gamma) \geq \sn(G) \geq \tw(G)$.
\end{corollary}

These lower bounds can be useful when computing gonalities.
However, once we resort to this, the benefit of using monoid-metrised graphs instead of unmetrised graphs is lost.

As a closing remark, we point out that Theorem \ref{thm:combinatorial-dgon} is not true if we omit the subdivision.
The following example shows that $\dgon(G)$ can be much larger than $\dgon(\Gamma)$.
\begin{example}
	Fix some $k \in \N_{\geq 1}$.
	Let $\Gamma'$ be the cycle of length $2k + 1$, where one edge $\{e_0,i(e_0)\}$ has weight $k$ and all other edges have weight $1$.
	Let $v$ be the vertex that is furthest away from $\{e_0,i(e_0)\}$, and let $w$ be one of the endpoints of $\{e_0,i(e_0)\}$.
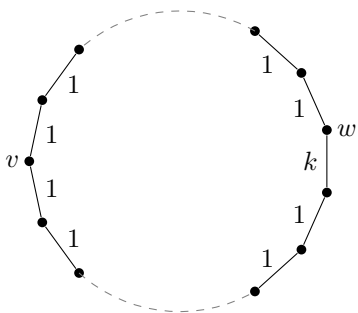
	\begin{center}
		\begin{tikzpicture}[vertex/.style={fill,circle,inner sep=1.3pt}]
			\begin{scope}[scale=2]
				\def\afst{.87}
				\foreach \x in {1,2,3,6,7,8,9,10,11,14,15} {
					\pgfmathsetmacro{\hoekV}{(\x - 1) / 15 * 360 + 180}
					\node[vertex] (v\x) at (\hoekV:1cm) {};
				}
				\foreach \x in {1,2,6,7,8,9,10,14,15} {
					\pgfmathtruncatemacro{\y}{mod(\x,15) + 1}
					\draw (v\x) -- (v\y);
					\pgfmathsetmacro{\hoekL}{(\x - .5) / 15 * 360 + 180}
					\pgfmathtruncatemacro{\lbl}{\x == 8 ? -1 : 1}
					\ifnum\lbl=1
						\node at (\hoekL:\afst cm) {\small$1$};
					\else
						\node at (\hoekL:\afst cm) {\small$k$};
					\fi
				}
				\draw[gray,dashed] (v3) arc (228:300:1cm);
				\draw[gray,dashed] (v11) arc (60:132:1cm);
				\draw (v1) node[left] {\small$v$};
				\draw (v9) node[right] {\small$w$};
			\end{scope}
		\end{tikzpicture}
	\end{center}
	The shorter arc between $v$ and $w$ contains $k$ edges of total length $k$, and the longer arc contains $k + 1$ edges of total length $2k$.
	It is not hard to see that the divisors $3[v]$ and $3[w]$ are equivalent on $\Gamma'$ and have rank $\geq 1$.
	
	Let $\Gamma$ be the $\N$-metrised graph obtained by gluing together $k$ copies of $\Gamma'$, where $w_i$ is identified with $v_{i+1}$ for all $i \in \{1,\ldots,k-1\}$.
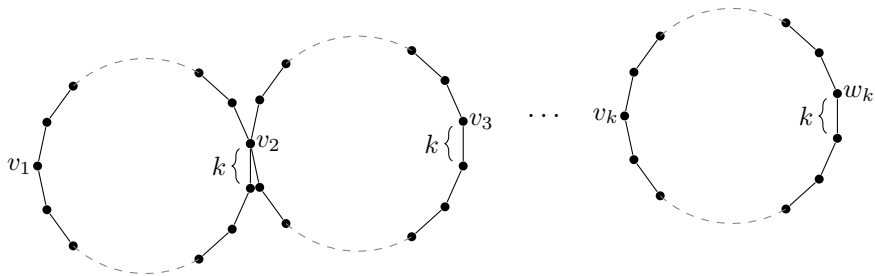
	\begin{center}
		\begin{tikzpicture}[vertex/.style={fill,circle,inner sep=1.2pt}]
			\begin{scope}[scale=1.43,inner sep=2pt]
				\def\afst{.84}
				\def\xsh{0.1}
				\begin{scope}
					\foreach \x in {1,2,3,6,7,8,9,10,11,14,15} {
						\pgfmathsetmacro{\hoekV}{(\x - 1) / 15 * 360 + 180}
						\node[vertex] (v\x) at (\hoekV:1cm) {};
					}
					\foreach \x in {1,2,6,7,8,9,10,14,15} {
						\pgfmathtruncatemacro{\y}{mod(\x,15) + 1}
						\draw (v\x) -- (v\y);
					}
					\draw[decorate,decoration={brace,amplitude=1mm}] ([xshift=-\xsh cm]v8.north) to node[left=\xsh cm] {\small$k$} ([xshift=-\xsh cm]v9.south);
					\draw[gray,dashed] (v3) arc (228:300:1cm);
					\draw[gray,dashed] (v11) arc (60:132:1cm);
					\draw (v1) node[left] {\small$v_1$};
				\end{scope}
				\begin{scope}[shift={(v9)},xshift=1cm]
					\foreach \x in {1,2,3,6,7,8,9,10,11,14,15} {
						\pgfmathsetmacro{\hoekV}{(\x - 1) / 15 * 360 + 180}
						\node[vertex] (v\x) at (\hoekV:1cm) {};
					}
					\foreach \x in {1,2,6,7,8,9,10,14,15} {
						\pgfmathtruncatemacro{\y}{mod(\x,15) + 1}
						\draw (v\x) -- (v\y);
					}
					\draw[decorate,decoration={brace,amplitude=1mm}] ([xshift=-\xsh cm]v8.north) to node[left=\xsh cm] {\small$k$} ([xshift=-\xsh cm]v9.south);
					\draw[gray,dashed] (v3) arc (228:300:1cm);
					\draw[gray,dashed] (v11) arc (60:132:1cm);
					\draw (v1) node[right] {\small$v_2$};
					\draw (v9) node[right] {\small$v_3$};
				\end{scope}
				\begin{scope}[shift={(v9)},xshift=2.5cm,yshift=.5mm]
					\node at (-1.75,0) {$\ldots$};
					\foreach \x in {1,2,3,6,7,8,9,10,11,14,15} {
						\pgfmathsetmacro{\hoekV}{(\x - 1) / 15 * 360 + 180}
						\node[vertex] (v\x) at (\hoekV:1cm) {};
					}
					\foreach \x in {1,2,6,7,8,9,10,14,15} {
						\pgfmathtruncatemacro{\y}{mod(\x,15) + 1}
						\draw (v\x) -- (v\y);
					}
					\draw[decorate,decoration={brace,amplitude=1mm}] ([xshift=-\xsh cm]v8.north) to node[left=\xsh cm] {\small$k$} ([xshift=-\xsh cm]v9.south);
					\draw[gray,dashed] (v3) arc (228:300:1cm);
					\draw[gray,dashed] (v11) arc (60:132:1cm);
					\draw (v1) node[left] {\small$v_k$};
					\draw (v9) node[right] {\small$w_k$};
				\end{scope}
			\end{scope}
		\end{tikzpicture}
	\end{center}
	Then clearly $3[v_1] \sim 3[v_2] \sim \cdots \sim 3[w_k]$, and this divisor has positive rank, so $\dgon(\Gamma) \leq 3$.
	On the other hand, the underlying graph $G$ is generic in the sense of \cite[Def.~4.1]{CoolsDraismaPayneRobeva2012}, so it follows from \cite[Thm.~1.1]{CoolsDraismaPayneRobeva2012} that $\dgon(G) = \lfloor \frac{k}{2} \rfloor + 1$.
	This shows that $\dgon(G)$ is not bounded by a function in $\dgon(\Gamma)$.
\end{example}

\appendix

\newcommand{\ghost}{\overline{M}}

\section{Logarithmic perspective}\label{sec:log_appendix}

The purpose of this section is to give an algebro-geometric interpretation of the refined notion of geometric gonality introduced in this paper. For this we assume familiarity with the language of log geometry (see \cite{Kato1989Logarithmic-str} or \cite{Ogus}), and with log curves (see \cite{Kato2000Log-smooth-defo}). 

Let $S$ be a geometric log point and $C/S$ a proper vertical log curve. We write $\ghost = \ghost_{S}(S)$ for the global sections of the characteristic monoid of $S$ and $\Gamma$ for the dual graph of $C$, which comes with a metric taking values in $\ghost$. 

Let $C'/S$ be another proper vertical log curve with graph $\Gamma'$, and let $f\colon C \to C'$ be a flat morphism over $S$. 

\begin{lemma}
The morphism $f$ induces a morphism of $\ghost$-metrised graphs $\frak F\colon \Gamma \to \Gamma'$ in the sense of Definition \ref{d_morpismm}. 
\end{lemma}
\begin{proof}
Irreducible components are mapped to irreducible components. A singular point $p$ is mapped either to a singular point, or to a smooth point (in which case all components through $p$ map to the irreducible component containing $f(p)$). It thus remains to check the divisibility condition on the edge lengths for non-contracted edges. 

Suppose $p' \coloneqq f(p)$ is another singular point, and write $l(p)$, $l(p') \in \ghost$ for the lengths of the associated edges; we must show that $l(p')$ is a positive integer multiple of $l(p)$. The characteristic monoids at $p$ and $p'$ are given by  
\begin{equation*}
    \ghost_{C, p} = \{(a,b) \in \ghost \times \ghost : a - b \in \langle l(p) \rangle\}
\end{equation*}
and
\begin{equation*}
    \ghost_{C', p'} = \{(a,b) \in \ghost \times \ghost : a - b \in \langle l(p') \rangle\}. 
\end{equation*}
The map $\ghost_{C', p'} \to \ghost_{C, p}$ induced by $f$ is a restriction of the identity on $\ghost \times \ghost$ (because the log structures on $C/S$ and $C'/S$ are strict on the smooth locus), hence $\langle l(p') \rangle \subseteq \langle l(p) \rangle$. 
\end{proof}

Now since $f\colon C \to C'$ is flat, it is also finite, and hence has a degree. 

\begin{lemma}
Let $p' \in C'$ be a singular point corresponding to an edge $e'$ of $\Gamma'$. Then the multiplicity of $\frak F$ at $e'$ is the same as the degree of $f$ at $p$. 
\end{lemma}
\begin{proof}
Every point $p$ of $C$ lying over $p'$ is singular, so it suffices to show that the slope of $\frak F$ at the corresponding edge $e$ of $\Gamma$ is equal to the ramification degree of $f$ at $p$. For this we use an alternative presentation of the characteristic monoids: at $p$ we have
\begin{equation*}
    \ghost_{C, p} = \ghost \oplus_{\mathbb N} \mathbb N^2
\end{equation*}
where $\mathbb N \to \ghost$ sends $1$ to $l(p)$ and $\mathbb N \to \mathbb N^2$ sends $1$ to $(1,1)$. 
Similarly at $p'$ we have
\begin{equation*}
    \ghost_{C', p'} = \ghost \oplus_{\mathbb N} \mathbb N^2
\end{equation*}
where this time $\mathbb N \to \ghost$ sends $1$ to $l(p')$. In particular the map $\mathbb N^2 \to \mathbb N^2$ induced by $f$ must send $(1,1)$ to $(s,s)$ where $s = l(p')/l(p)$ is the slope. Hence the ramification degrees of the two branches are equal to each other and to the slope. 
\end{proof}

From the previous lemma we immediately deduce
\begin{lemma}
Suppose that $C'$ is not smooth. Then $\frak F\colon \Gamma \to \Gamma'$ is horizontally conformal, of degree equal to the degree of $f$. 
\end{lemma}

To complete our connection to the graph-theoretic geometric gonality we observe 
\begin{lemma}
Suppose that $C'$ has arithmetic genus 0. Then $\Gamma'$ is a tree. 
\end{lemma}

\begin{proposition}
Suppose that $C'$ has arithmetic genus 0 and is not smooth. Then the degree of $f$ is greater than or equal to the geometric gonality of the $\ghost$-metrized graph $\Gamma$. 
\end{proposition}
\begin{proof}
The map $\frak F\colon \Gamma \to \Gamma'$ is a harmonic morphism to a tree, and has degree equal to the degree of $f$. 
\end{proof}

Using that the degree is locally constant in flat families we can deduce a `global' version of the above proposition, which can be thought of as a `specialization lemma' for monoid-metrised graphs. 

\begin{theorem}\label{log_theorem}
Let $S$ be a connected log scheme, and let $f\colon C\to C'$ be a flat $S$-morphism of proper vertical log curves, with $C'$ of arithmetic genus 0. Let $s \in S$ be a geometric point with $C'_s$ singular, and write $\ggon(s)$ for the geometric gonality of the graph associated to $C_s$, and $\gon(s)$ for the divisorial gonality. Then
\begin{equation*}
\deg f \ge \ggon(s)\ge \gon(s). 
\end{equation*}
\end{theorem}



%

\begin{remark}\label{log_eg}
We give a geometric interpretation of this result in the case of the 2-gon over $\mathbb N^2$ (Example \ref{e_PLconstant}). In particular this example shows that the monoid-metrised gonality gives sharper bounds on geometric gonalities than previous approaches. 

Firstly, if one ignores the metrics, there is a unique harmonic degree 2 map to the tree with one edge. This reflects a geometric fact. Define a curve $C$ by glueing two copies of the projective line $\mathbb P^1$ and the points $(1:1)$ and $(1:-1)$, and a curve $D$ by glueing two copies of $\mathbb P^1$ just at $(1:1)$. Then we can make a degree-$2$ map $$f\colon C \to D$$ where on the first component we take the squaring map to the first $\mathbb P^1$, and on the second component the squaring map to the second. 

If one works with classically metrised graphs (both edges being given length 1), then this harmonic degree 2 map still exists. On a geometric level, this reflects the fact that a very carefully-chosen 1-parameter deformation of $C$ admits a map to a 1-parameter deformation of $D$.

Finally, suppose we equip the graph with the $\mathbb N^2$-metric from Example \ref{e_PLconstant}. Then the geometric gonality is infinite (see Remark \ref{rem:infinite_gonality}), so by Theorem \ref{log_theorem} a corresponding morphism of log curves does not exist. This implies that the map $f$ does not extend to a `generic' deformation of the curve $C$.
\end{remark}


\begin{thebibliography}{99}
    
\bibitem{Aidun2018} 
    {\scshape Aidun, Ivan; Dean, Frances; Morrison, Ralph; Yu, Teresa; Yuan, Julie.} 
    Graphs of gonality three, 
    \textit{Alg. Comb. 2} 
    \textbf{2} 
    \text{(2018) 1197-1217}. 
    \zbl{1441.14203}.
    
\bibitem{AminiCaporaso2011} 
    {\scshape Amini, Omid; Caporaso, Lucia.} 
    Riemann-Roch theory for weighted graphs and tropical curves, 
    \textit{Adv. Math.} 
    \textbf{240} 
    \text{(2011), 1-23}. 
    \zbl{1284.14087}.
    
    
    \bibitem{Baker2008Specialization}
    {\scshape Baker, Matthew.}
    Specialization of linear systems from curves to graphs,
    \textit{Algebra Number Theory}
    \textbf{6}
    \text{(2008), 613-653}
    \zbl{1162.14018}.



\bibitem{BakerNorine2007HM} 
    {\scshape Baker, Matthew; Norine, Serguei.} 
    Harmonic morphisms ans hyperelliptic graphs, 
    \textit{Int. Math. Res. Not.} 
    \textbf{2009} 
    \text{(2009), 2914-2955}. 
    \zbl{1178.05031}.
 
\bibitem{BakerNorine2007RR} 
    {\scshape Baker, Matthew; Norine, Serguei.} 
    Riemann-{R}och and {A}bel-{J}acobi theory on a finite graph, 
    \textit{Advances in Mathematics} 
    \textbf{215} 
    \text{(2007), 766-788}. 
    \zbl{1124.05049}.
    
\bibitem{Caporaso2014} 
    {\scshape Caporaso, Lucia.} 
    Gonality of algebraic curves and graphs, 
    \textit{Springer Proceedings in Mathematics} 
    \textbf{71} 
    \text{(2014), 77-108}. 
    \zbl{1395.14026}.
    
\bibitem{Chan2012} 
    {\scshape Chan, Melody.} 
    Tropical hyperelliptic curves, 
    \textit{J. Algebr. Comb.} 
    \textbf{37} 
    \text{(2013), 331-359}. 
    \zbl{1266.14050}.
    
\bibitem{CoolsDraisma2018} 
    {\scshape Cools, Philip; Draisma, Jan.} 
    On metric graphs with prescribed gonality, 
    \textit{J. Comb. Theory} 
    \textbf{156} 
    \text{(2018), 1-21}. 
    \zbl{1381.05012}.
    
\bibitem{CoolsDraismaPayneRobeva2012} 
    {\scshape Cools, Philip; Draisma, Jan; Payne, Sam; Robeva, Elina.} 
    A tropical proof of the {B}rill-{N}oether theorem, 
    \textit{Adv. Math} 
    \textbf{230} 
    \text{(2012), 759-776}. 
    \zbl{1325.14080}.
    
    \bibitem{HarpJacksonJensenSpeeter2022}
    {\scshape Harp, Michael; Jackson, Elijah; Jensen, David; Speeter, Noah.}
    A new lower bound on graph gonality,
    \textit{Discrete Appl. Math.}
    \textbf{309}
    \text{(2022), 172-179}.
    \zbl{7456389}.
    
\bibitem{Kageyama2018} 
    {\scshape Kageyama, Yuki.} 
    Divisorial condition for the stable gonality of tropical curves, 
    \newblock {\em {arXiv preprint \url{https://arxiv.org/abs/1801.07405}}}, 2018.

\bibitem{Kato1989Logarithmic-str}
{\scshape Kato, Kazuya. }
\newblock Logarithmic structures of {F}ontaine-{I}llusie.
\newblock In {\em Algebraic analysis, geometry, and number theory ({B}altimore,
  {MD}, 1988)}, pages 191--224. Johns Hopkins Univ. Press, Baltimore, MD, 1989.
  
\bibitem{Kato2000Log-smooth-defo}
{\scshape Kato, Fumiharu.}
\newblock Log smooth deformation and moduli of log smooth curves.
\newblock {\em Internat. J. Math.}, 11(2):215--232, 2000.
    
    
\bibitem{Luo11}
    {\scshape Luo, Ye.}
    Rank-determining sets of metric graphs.
    \textit{J. Comb. Theory, Ser. A}
    \textbf{118}
    \text{(2011), 1775–1793}.
    \zbl{1227.05133}


\bibitem{Molcho2018} 
    {\scshape Molcho, Samouil; Wise, Jonathan.} 
    The logarithmic {P}icard group and its tropicalization, 
    (2018)
\newblock {\em {arXiv preprint \url{https://arxiv.org/abs/1807.11364}}}, 2018.


\bibitem{Ogus} 
    {\scshape Ogus, Arthur;}    
    Lectures on Logarithmic Algebraic Geometry, 
    {\em Cambridge University Press}, 2018 
    
\bibitem{VanderWegen2017} 
    {\scshape van der Wegen, Marieke.} 
    Stable gonality of graphs, 
    \textit{Master's thesis, Utrecht University, }(2017)
    
\bibitem{VanDobbenDeBruyn2012} 
    {\scshape van Dobben de Bruyn, Josse.} 
    Reduced divisors and gonality in finite graphs, 
    \textit{Bachelor's thesis, Universiteit Leiden, }(2012)
 
 \bibitem{VanDobbenDeBruynGijswijt2020}
    {\scshape van Dobben de Bruyn, Josse; Gijswijt, Dion.}
    Treewidth is a lower bound on graph gonality,
    \textit{Algebr. Comb.} 
    \textbf{3} 
    \text{(2020), 941--953}. 
    \zbl{7251040}.


\end{thebibliography}
\end{document}